\documentclass[11pt]{amsart}
\newcommand\A{\mathfrak A}
\newcommand\supp{\operatorname{supp}}
\newcommand\mor{\operatorname{mor}}
\newtheorem*{theorem*}{Theorem}
\newtheorem{lemma}{Lemma}
\begin{document}
\title{Non-isomorphism of categories of algebras}
\author{Walter D. Neumann}
\begin{abstract}This paper was accepted for Comment.\ Math.\ Univ.\
  Carolinae in 1968 but then got lost during the military occupation
  of  Prague and surrounding events. A carbon copy of it turned up
  in my Columbia office.  The only changes to the text are
  addition of a footnote and some commas.\end{abstract}\maketitle

We recall that
if $\Delta = (K_i)_{i\in I}$ is a collection of sets, then an algebra
$(A,(f_i)_{\{i\in I\}})$ of type $\Delta$ consists of a set $A$ and a
collection $(f_i)_{i\in I}$ of operations on $A$, such that for each
$i\in I$ $f_i$ is a $K_i$--ary operation; that is, $f_i\colon
A^{K_i}\to A$. We shall use the same symbol for an algebra $A$ and its
underlying set.

We say that the types $\Delta=(K_i)_{i\in I}$ and $\Delta'=(L_j)_{j\in
  J}$ are \emph{equivalent} if there is a bijection $\phi\colon I\to
  J$ such that for each $i$ the cardinalities $|K_i|$ and
  $|L_{\phi(i)}|$ are equal.

We denote by $\A(\Delta)$ the category of all algebras of type
$\Delta$ and all homomorphisms between them. Clearly, if $\Delta$ and
$\Delta'$ are equivalent types, then $\A(\Delta)$ and $\A(\Delta')$
are isomorphic categories. The purpose of this note is to prove the
converse, thereby answering a question posed by A. Pultr.

This question was motivated by the result of Z. Hedrl\'in and A. Pultr
\cite{2} which states that even if $\Delta$ and $\Delta'$ are not
equivalent, $\A(\Delta)$ and $\A(\Delta')$ are embeddable as full
subcategories in each other, so long as $\Delta$ and $\Delta'$ are not
too small (the sums of the cardinalities of the sets in $\Delta$ and
$\Delta'$ should each exceed $1$).

In \S 1 we recall the necessary properties of algebraic operations,
and in \S 2 we prove the theorem.

\section{Operations}

Let $f\colon F^K\to K$ be a $K$--ary operation on $F$ ($F, K$ any
sets). The \emph{support} of $f$ (Felscher \cite{1}) is defined by
$$\supp(f)=\{A\subseteq K~|~ \text{for all }\alpha,\beta\in F^K,
\alpha|A=\beta|A \Rightarrow f(\alpha)=f(\beta)\}$$
That is, $A\in \supp(f)$ means that the value of $f$ on any $\alpha\in
F^K$ is already determined by the restriction of $\alpha$ to $A$.

The \emph{essential rank} of $f$ is defined as $\min\{|A|~|~A\in
\supp(f)\}$. If $\A$ is a primitive class (variety) of algebras,
define its \emph{rank} to be the supremum of the essential ranks of
all $\A$--algebraic operations. Since the essential rank of an
algebraic operation is always less than the dimension (S\l omi\'nski
\cite{3}) of $\A$, this supremum exists.

Now let $f$ be any element of the free algebra $F(X,\A)$ with basis
$X$ of the primitive class $\A$, and let $A$ be any algebra in
$\A$. One can define an $X$--ary operation $\hat f^A$ on $A$ by $\hat
f^A(\alpha)=\bar\alpha(f)$ for any $\alpha\in A^X$. Here $\bar\alpha$
is the homomorphic extension of $\alpha$ to a homomorphism of
$F(X,\A)$ to $A$. $\hat f^A$ is an algebraic operation on $A$; in fact
$f\mapsto \hat f^A$ is a surjective homomorphism of $F(X,\A)$ onto the
algebra $H^X(A)$ of all $X$--ary algebraic operations on $A$ (see for
instance \cite{1}).

\section{Non-isomorphism of categories}
\begin{theorem*}
  If the categories $\A(\Delta)$ and $\A(\Delta')$ are isomorphic, then
  $\Delta$ and $\Delta'$ are equivalent types.
\end{theorem*}

We shall prove this theorem in two steps. Let $s$ denote the canonical
underlying set functors on $\A(\Delta)$ and $\A(\Delta')$.

\begin{lemma}
  If the concrete categories $(\A(\Delta),s)$ and $(\A(\Delta'),s)$
  are concretely isomorphic (that is, isomorphic by a functor which
  preserves underlying sets), then $\Delta$ is equivalent to $\Delta'$.
\end{lemma}

\begin{proof}
  Let $\Delta=(K_i)_{i\in I}$. It suffices to show that the knowledge
  of the category $\A(\Delta)$ and its underlying set functor is
  sufficient to recover $\Delta$ up to equivalence. Now the definition
  of free algebra of $\A(\Delta)$ over a basis involves only the
  underlying sets and homomorphisms, so we know the free algebras of
  $\A(\Delta)$. Hence by the last paragraph of \S1 we know the
  algebraic operations for $\A(\Delta)$ so we can calculate the rank
  $\delta$ of $\A(\Delta)$.

Let $F$ be a free algebra of $\A(\Delta)$ with basis $X$ such that
$|X|\ge \delta$. Let $Y=F-X$. Define a relation $R$ on $Y$ by
$y_1Ry_2$ if and only if there is an endomorphism of $F$ which maps
$y_1$ onto $y_2$. Let $S$ be the smallest equivalence relation on $Y$
containing $R$, and let $J=Y/S$ be the set of equivalence classes of
$Y$ under $S$.

We now consider the meaning of this construction in terms of the
actual algebraic structure of $F$. The set $Y$ is the set of all
$f_i(\alpha)$, $i\in I, \alpha\in F^{K_i}$, where the $f_i$ are the
defining operations of the class. Under an endomorphism of $F$ an
element $f_i(\alpha)$ cannot be mapped onto an element $f_j(\beta)$
with $i\ne j$. Further, if $\alpha_0\in F^{K_i}$ is injective with
$\alpha_0(K_i)\subseteq X$, then every element of the form
$f_i(\alpha)$, $\alpha\in F^{K_i}$, is the image of $f_i(\alpha_0)$
under  suitable endomorphisms of $F$. Hence the equivalence classes
under $S$ are just the subsets of $Y$ of the form
$\phi(i)=\{f_i(\alpha)~|~\alpha\in F^{K_i}\}$, and $\phi\colon
i\mapsto \phi(i)$ is a bijective map from $I$ to $J$.

It remains only to show that to each $\phi(i)\in J$ we can recover
$|K_i|$. We can choose a $y\in \phi(i)$ with the property that every
element of $\phi(i)$ is the image of $y$ under some endomorphism of
$F$. We then have the corresponding $X$--ary algebraic operation $\hat
y=\hat y^F$ in $H^X(F)$. We claim that $\hat y$ has essential rank
$|K_i|$, completing the proof. Indeed, $y$ is of the form
$f_i(\alpha_0)$ for some injective $\alpha_0\in F^{K_i}$ with
$\alpha_0(K_i)\subseteq X$. $\supp(\hat y)$ is just the set of all
subsets of $X$ which contain $\alpha_0(K_i)$, so the essential rank of
$\hat y$ is $|\alpha_0(K_i)|$. Since $\alpha_0$ is injective,
$|\alpha_0(K_i)|=|K_i|$.
\end{proof}

The second part of the proof of the theorem is given by the following
lemma:

\begin{lemma}\label{2}
  If the categories $\A(\Delta)$ and $\A(\Delta')$ are isomorphic,
  then the concrete categories $(\A(\Delta),s)$ and $(\A(\Delta'),s)$
  are concretely isomorphic.
\end{lemma}
\begin{proof}
  We first determine a free algebra $P$ of rank $1$ (that is, basis
  cardinality 1) in $\A(\Delta)$. This can be done in many ways, For
  instance, $P\in \A(\Delta)$ is a free algebra if and only if every
  epimorphism to $P$ has a section, and it furthermore has rank $1$ if
  and only if every morphism $P\to P$ is mono.

The functor $\mor(P,-)\colon \A(\Delta)\to \bold{Set}$ is a ``new
underlying set functor'' which is naturally equivalent to the standard
underlying set functor $s$ on $\A(\Delta)$.

Let $T\colon \A(\Delta)\to \A(\Delta')$ be an isomorphism. Then
$T(P)$ is a free algebra of rank $1$ in $\A(\Delta')$, so we also have
a ``new underlying set functor'' $\mor(T(P),-)$ on $\A(\Delta')$ which
is naturally equivalent to the standard one. 

If one identifies the sets $\mor(P,A)$ and $\mor(T(P),T(A))$ by means
of $T$ for each $A\in \A(\Delta)$, then $T$ commutes with these ``new
underlying set functors'', and it is not difficult, using the
properties of the standard underlying set functors, to deduce that
$\A(\Delta)$ and $\A(\Delta')$ are also concretely isomorphic with
respect to the standard underlying set functors.
\end{proof}

Lemma \ref{2} states that two categories of the form $\A(\Delta)$ are
abstractly isomorphic if and only if they are concretely isomorphic,
or in more algebraic terminology: rationally equivalent. This in fact
holds for more general primitive classes of algebras; for instance a
slight modification of the above proof shows that two Schreier
primitive classes (subalgebras of free algebras are free) whose free
algebras of rank 1 are not isomorphic to any of higher
rank\footnote{The phrase ``whose free algebras \dots higher rank'' was a
  handwritten addition in the 1968 typescript -- apparently an
  afterthought.} are abstractly isomorphic as
categories if and only if they are rationally equivalent. Some
restriction on the classes considered is however necessary, for it is
known that to any primitive class one can find primitive classes not
rationally equivalent to the given one, such that the categories are
isomorphic.

\end{document}